\documentclass[12pt]{amsart}
\title[Class groups of open Richardson varieties]
{Class groups of open Richardson varieties in the Grassmannian are trivial}

\author{Jake Levinson}
\address{Mathematics Department\\
         University of Washington\\
         Seattle, WA, 98195-4350\\
         USA}
\email{jlev@uw.edu}

\author{Kevin Purbhoo}
\address{Combinatorics \& Optimization Dept.\\
         University of Waterloo\\
         Waterloo, ON, N2L 3G1\\
         CANADA}
\email{kpurbhoo@math.uwaterloo.ca}

\usepackage{latexsym, amsmath, amscd}
\usepackage[shortlabels]{enumitem}
\usepackage{amsthm}
\usepackage{graphicx, xspace, ifthen, rotating, pict2e}
\usepackage[svgnames]{xcolor}
\usepackage[charter]{mathdesign}
\usepackage{algpseudocode, algorithmicx}
\usepackage{tikz}

\allowdisplaybreaks[1]

\newcommand{\ZZ}{\mathbb{Z}}

\newcommand{\RR}{\mathbb{R}}

\newcommand{\calI}{\mathcal{I}}
\newcommand{\calM}{\mathcal{M}}

\newcommand{\calP}{\mathcal{P}}
\newcommand{\calQ}{\mathcal{Q}}

\newcommand{\openX}{\mathring{X}}
\newcommand{\openA}{\mathring{A}}

\newcommand{\nth}{\ensuremath{^\text{th}}\xspace}
\newcommand{\Plucker}{Pl\"ucker\xspace}

\newcommand{\Gr}{\mathrm{Gr}}

\newcommand{\Proj}{\mathop{\mathrm{Proj}\,}}
\newcommand{\Spec}{\mathop{\mathrm{Spec}\,}}

\newcommand{\allsubsets}{\mathcal{S}(k,n)}
\newcommand{\pc}{\Delta}
\newcommand{\symgroup}{\mathfrak{S}_n}
\newcommand{\lowerschubert}{X_\beta}
\newcommand{\openlowerschubert}{\openX_\beta}
\newcommand{\upperschubert}{X^\gamma}
\newcommand{\openupperschubert}{\openX^\gamma}
\newcommand{\rich}{X_\beta^\gamma}
\newcommand{\openrich}{\openX_\beta^\gamma}
\newcommand{\pdivisor}{X_{\calP_t}}
\newcommand{\openpdivisor}{\openX_{\calP_t}}
\newcommand{\matroidvariety}[1]{X_{#1}}
\newcommand{\paramrich}{W_\beta^\gamma}
\newcommand{\paramrichmat}{Y_\beta^\gamma}
\newcommand{\lowerschubertring}{A_\beta}
\newcommand{\openlowerschubertring}{\openA_\beta}
\newcommand{\upperschubertring}{A^\gamma}
\newcommand{\openupperschubertring}{\openA^\gamma}
\newcommand{\richring}{A_\beta^\gamma}
\newcommand{\openrichring}{\openA_\beta^\gamma}
\newcommand{\paramrichring}{B_\beta^\gamma}

\newcommand{\grring}{A}
\newcommand{\pluckerideal}{I}
\newcommand{\field}{\mathbb{K}}

\newenvironment{packedenumi}{
\begin{enumerate}[(i)]
  \setlength{\itemsep}{0pt}
}{\end{enumerate}}

\newenvironment{packeditemize}{
\begin{itemize}[topsep=1ex]
  \setlength{\itemsep}{0pt}
  \setlength{\parskip}{.4ex}
}{\end{itemize}}

\newtheorem{lemma}{Lemma}%[section]
\newtheorem{theorem}[lemma]{Theorem}

\newtheorem{proposition}[lemma]{Proposition}
\newtheorem{corollary}[lemma]{Corollary}

\theoremstyle{definition}

\newtheorem{remark}[lemma]{Remark}

\begin{document}

\begin{abstract}
We prove that the divisor class group of any open Richardson variety in the Grassmannian is trivial.  Our proof uses Nagata's criterion, localizing the coordinate ring at a suitable set of \Plucker coordinates. We prove that these \Plucker coordinates are prime elements by showing that the subscheme they define is an open subscheme of a positroid variety. Our results hold over any field and over the integers.
\end{abstract}

\maketitle

%%%%%%%%%%%%%
%%%%%%%%%%%%%
\section{Introduction}

Let $\openrich \subset \Gr(k,n)$ be an open Richardson variety in the Grassmannian, over an arbitrary field or the integers. The purpose of this paper is to prove the following theorem.

\begin{theorem}
\label{thm:classgroup}
The divisor class group of $\openrich$ is trivial.
\end{theorem}

The corresponding statement for Schubert cells is clear, since every Schubert cell is isomorphic to affine space.  The divisor class group and the Picard group of Schubert varieties in a generalized flag manifold (over a field) were computed by Mathieu \cite{Mathieu} (see also \cite{Brion}, for an overview of these and other related results). This is made even more explicit in \cite{WY}.

Comparatively little appears to be known about the divisor class group and the Picard group of Richardson varieties. Open Richardson varieties are smooth, affine, and rational, but these facts do not determine the divisor class group. On the other hand, each of the following statements is equivalent to Theorem~\ref{thm:classgroup} (see e.g. \cite[\S 14.2]{Vakil}):
\begin{packeditemize}
\item Every divisor on $\openrich$ is principal.
\item The coordinate ring of $\openrich$ is a unique factorization domain.
\item The Picard group of $\openrich$ is trivial.
\item The divisor class group of the closed Richardson variety 
$\rich$ is generated
by the boundary divisors of $\openrich$.
\end{packeditemize}

We prove Theorem~\ref{thm:classgroup} by showing that the homogeneous coordinate ring is a unique factorization domain. 
This is done by a suitable application of Nagata's criterion. As part of our argument, we study the principal ideals defined by certain \Plucker coordinates, $\pc_{\delta_1}, \dots, \pc_{\delta_{k-1}}$. We show that the subscheme of $\openrich$ defined by each of these ideals is an open subscheme of a positroid variety.  This allows us to deduce that these \Plucker coordinates are prime elements of the coordinate ring of the open Richardson variety.

\begin{figure}
\centering
\[
\begin{bmatrix}
%1   2   3   4   5   6   7   8   9  10  11  12  13  14
 0 & *'& * & * & * & 1 & 0 & 0 & 0 & 0 & 0 & 0 & 0 & 0 \\
 0 & 0 & 0 & 0 & *'& * & * & 1 & 0 & 0 & 0 & 0 & 0 & 0 \\
 0 & 0 & 0 & 0 & 0 & *'& * & * & * & * & 1 & 0 & 0 & 0 \\
 0 & 0 & 0 & 0 & 0 & 0 & 0 & 0 & 0 & *'& * & 1 & 0 & 0  \\
\end{bmatrix}
\,.
\]
\caption{Matrices parameterizing the subscheme $\paramrich \subseteq \Gr(4,14)$.  
Here $\beta=\{2,5,6,10\}$, $\gamma = \{6,8,11,12\}$.
The $*$-entries can assume any value, and $*'$ are required to be invertible.}
\label{fig:parameterized}
\end{figure}

There is an open subscheme $\paramrich \subseteq \openrich$, parameterized by $k \times n$ matrices of the form in Figure~\ref{fig:parameterized}.  $\paramrich$ is obtained by localizing the coordinate ring of $\openrich$ at $\pc_{\delta_1}, \dots, \pc_{\delta_{k-1}}$, which amounts to a statement about existence of certain $LU$ decompositions. As a consequence, we also obtain a description of this nicely parameterizable subscheme as the complement of a union of positroid varieties in the Richardson variety (see \S\ref{sec:remarks}).

It is important for us that Theorem \ref{thm:classgroup} holds without working over an algebraically closed field, because this paper was written with a specific application in mind.  For one of the results in \cite{LP}, we needed to show that a certain divisor on $\openrich$ is principal over $\RR$. This was not obvious from the definition of the divisor, and Theorem~\ref{thm:classgroup} arose as the most expeditious way to prove it.

It would be interesting to generalize Theorem~\ref{thm:classgroup} beyond the Grassmannian. David Speyer has suggested to us that it may be possible prove an analogous statement for flag varieties, using the Deodhar cell \cite{Deodhar} in place of $\paramrich$. This would be more complicated and would require some different tools, as the part of this story involving positroid varieties is necessarily specific to the Grassmannian.

%%%%%%%%%%%%%
%%%%%%%%%%%%%
\section{Subschemes of the Grassmannian}

Let $\allsubsets$ denote the set of all $k$-element subsets of $[n] = \{1, \dots, n\}$.  
For $\alpha \in \allsubsets$, write $\alpha(i)$ for the $i$\nth smallest element of $\alpha$.  There is a natural poset structure of $\allsubsets$ defined by $\alpha \leq \alpha'$ if $\alpha(i) \leq \alpha'(i)$ for all $i =1, \dots, k$.  For $\beta, \gamma \in \allsubsets$, write $[\beta, \gamma]$ to denote the interval $\{\alpha \in \allsubsets \mid
\beta \leq \alpha \leq \gamma\}$ in this poset.  We will also use this notation to denote intervals in other posets.

Let $\field$ be a field. The \emph{Grassmannian} $\Gr(k,n)$ is the projective scheme whose points represent $k$-planes in $\field^n$.  Formally, $\Gr(k,n) = \Proj \grring$, 
\[
 \grring = \field[\pc_\alpha \mid \alpha \in \allsubsets]/\pluckerideal
\,,
\]
where $\pc_\alpha$ are indeterminates called the \emph{\Plucker coordinates}, and $\pluckerideal$ is the \emph{\Plucker ideal}. $\pluckerideal$ is generated by quadratic elements of the form
\begin{equation}
\label{eqn:pluckerrelation}
   \pc_\alpha \pc_\beta 
  - \sum_{j \in \alpha \setminus \beta} \pm \pc_{\alpha \setminus \{j\} \cup \{i\}} \pc_{\beta \setminus \{i\} \cup \{j\}}
\,,
\end{equation}
where $\alpha, \beta \in \allsubsets$ and $i \in \beta \setminus \alpha$. The precise signs in~\eqref{eqn:pluckerrelation} are not too hard to describe, but we will not do so here, since we will not need them.

If $V \in \Gr(k,n)$ is a $k$-plane in $\field^n$, we can represent $V$ as the row space of a $k \times n$ matrix $M$ with entries in $\field$.  The \Plucker coordinates of $V$ are the maximal minors of this matrix: $V$ is associated to the map $\pc_\alpha \mapsto \det M_\alpha$, where $M_\alpha$ the sumatrix of $M$ with column set $\alpha$. Many important subschemes of $\Gr(k,n)$ are defined by vanishing and non-vanishing of certain \Plucker coordinates. These include Schubert varieties, Richardson varieties, and their
open counterparts.

Fix $\beta \leq \gamma \in \allsubsets$, and let
\begin{align*}
\lowerschubertring &= \grring/\langle \pc_\alpha \mid \alpha \ngeq \beta \rangle &
\lowerschubert &= \Proj \lowerschubertring
\\
\upperschubertring &= \grring/\langle \pc_\alpha \mid \alpha \nleq \gamma \rangle &
X^\gamma &= \Proj \upperschubertring
\\
\richring &= \grring/\langle \pc_\alpha \mid \alpha \notin [\beta,\gamma] \rangle&
\rich &= \Proj \richring
\,.
\end{align*}
$\lowerschubert$ is a \emph{Schubert variety}, $X^\gamma$ is an \emph{opposite Schubert variety}, and $\rich = \lowerschubert \cap X^\gamma$ is a \emph{Richardson variety}.  (See \cite{HP} for the equivalence of this definition with other definitions of Schubert varieties.)

The Schubert cells and the open Richardson variety are obtained algebraically as localizations of the rings above. For a ring $R$ and an element $r \in R$, let $R[r^{-1}]$ denote the localization of $R$ at $r$. Let
\begin{align*}
\openlowerschubertring &= \lowerschubertring[\pc_\beta^{-1}]  &
\openlowerschubert &= \Proj \openlowerschubertring
\\
\openupperschubertring &= \upperschubertring[\pc_\gamma^{-1}]  &
\openX^\gamma &= \Proj \openupperschubertring
\\
\openrichring &= 
\richring [\pc_\beta^{-1}, \pc_\gamma^{-1}] &
\openrich &= \Proj \openrichring
\,.
\end{align*}
Then $\openlowerschubert \subseteq \lowerschubert$ is the \emph{Schubert cell}, $\openX^\gamma \subseteq X^\gamma$ is the \emph{opposite Schubert cell}, and $\openrich = \openlowerschubert \cap \openX^\gamma$ is the \emph{open Richardson variety}.  Note that these are affine schemes, since $\Proj \grring[\pc_\alpha^{-1}] = \Spec \grring/\langle \pc_\alpha-1 \rangle$ (c.f. Remark~\ref{rmk:proj}). We will establish Theorem~\ref{thm:classgroup} by proving the following equivalent proposition.

\begin{proposition}
\label{prop:UFD}
$\openrichring$ is a unique factorization domain.
\end{proposition}

The proof of Proposition~\ref{prop:UFD} is based on Nagata's Criterion (see e.g. \cite[Lemma 19.20]{Eisenbud}).

\begin{lemma}[Nagata's Criterion]
\label{lem:Nagata}
Suppose $R$ is a Noetherian domain, and $S \subset R$ is a multiplicative subset generated by prime elements of $R$.  $R$ is a unique factorization domain if and only if $S^{-1}R$ is a unique factorization domain.
\end{lemma}

We apply this as follows. Define $\delta_0, \dots, \delta_k \in \allsubsets$ by
\[
   \delta_t = \{\beta(1), \dots, \beta(t), \gamma(t+1), \dots, \gamma(k)\}
\,.
\]
In particular, $\delta_0 = \gamma$ and $\delta_k = \beta$. Consider the ring
\[
\paramrichring
= \richring [\pc_{\delta_0}^{-1}, \dots, \pc_{\delta_k}^{-1}]
= \openrichring[\pc_{\delta_1}^{-1}, \dots, \pc_{\delta_{k-1}}^{-1}]
\,.
\]
\begin{proof}[Proof of Proposition~\ref{prop:UFD}]
We will show that $\paramrichring$ is a unique factorization domain (Corollary~\ref{cor:localizedufd}).  Moreover, $\pc_{\delta_1}, \dots \pc_{\delta_{k-1}}$ are either units or primes of $\openrichring$ (Corollaries~\ref{cor:unit} and~\ref{cor:prime}). The theorem then follows from Lemma~\ref{lem:Nagata}.
\end{proof}

\begin{remark} 
\label{rmk:proj}
In many references (e.g. \cite{Hartshorne, Vakil}), the $\Proj$ construction is defined only for non-negatively graded rings $R = \bigoplus_{d \geq 0} R_d$. However, the construction can be applied more generally to $\ZZ$-graded rings $R = \bigoplus_{d \in \ZZ} R_d$, in which negative degrees are allowed.
As $\bigoplus_{d > 0} R_d$ is not necessarily an ideal of $R$, the irrelevant ideal $R_+$ is defined as the ideal generated by homogeneous elements of positive degree; this is the only change required.  If $R$ is a $\field$-algebra, this is equivalent to defining $\Proj R$ as the GIT quotient $\Spec R/\!\!/\mathbb{G}_m$, for the linearized $\mathbb{G}_m$ action on $\Spec R$ defined by the grading (see \cite[\S8.3]{Dolgachev} for examples).  In the $\ZZ$-graded rings above ($\openlowerschubertring$, $\openupperschubertring$, etc.), the irrelevant ideal is $\langle 1 \rangle$, which implies $\Proj R = \Spec R_0$, and $\Spec R = \Proj R \times \mathbb{G}_m$. We make use of this implicitly in our arguments.
\end{remark}

\subsection{Working over $\Spec \mathbb{Z}$}

The rings and schemes of the previous section are also defined for $\field = \mathbb{Z}$. We briefly say how to extend Proposition \ref{prop:UFD} and Theorem \ref{thm:classgroup} to this case.

\begin{lemma} \label{lem:integral-over-z}
Over $\Spec \mathbb{Z}$, $\rich$ (and therefore $\openrich$) is an integral scheme.
\end{lemma}
\begin{proof}
The corresponding statement over an arbitrary field is well-known. Therefore, if $\rich$ is not integral over $\mathbb{Z}$, the ring $\richring$ must have $p$-torsion for some $p$. This cannot happen since e.g. the Hilbert function of $\richring$ is the same over any field.
\end{proof}

\begin{lemma}
If Proposition \ref{prop:UFD} holds over $\mathbb{Q}$, it holds over $\mathbb{Z}$.
\end{lemma}
\begin{proof}
By Lemma \ref{lem:integral-over-z}, $\openrichring$ is an integral domain. Since $\mathbb{F}_p \otimes \openrichring$ is again integral, the integers $p \in \openrichring$ are prime elements. Thus, by Nagata's Criterion, $\openrichring$ is a UFD if and only if $\mathbb{Q} \otimes \openrichring$ is a UFD.
\end{proof}
In fact, most of the intermediate results of this paper are valid with $\field = \mathbb{Z}$. The exception is Theorem \ref{thm:KLS} on integrality (in particular, torsion-freeness) of positroid varieties. This step would follow if their Hilbert functions were known to be the same over any field (e.g. by Grobner basis arguments).

For the remainder of the paper, we work over an arbitrary field.

%%%%%%%%%%%%%
%%%%%%%%%%%%%
\section{A parameterized open subscheme}

Consider the open subscheme of $\openrich$,
\[
   \paramrich 
   = \Proj \paramrichring = \Spec \paramrichring/\langle \pc_\gamma -1\rangle
\,.
\]
In this section, we show that $\paramrich$ is parameterized by matrices of the form in Figure~\ref{fig:parameterized}; in particular, it is a product of affine space and a torus. We thereby deduce that $\paramrichring$ is a unique factorization domain.

Let $\paramrichmat$ denote the space of matrices $M$ shown in Figure~\ref{fig:parameterized}. Explicitly, for each $i$ we have $M_{i, \beta(i)}$ invertible, $M_{i, \gamma(i)} = 1$, $M_{i,j}$ arbitrary for $\beta(i) < j < \gamma(i)$, and $M_{i,j} = 0$ otherwise.

\begin{theorem}
\label{thm:paramrich}
The natural map $\phi: \paramrichmat \to \paramrich$
(defined by the ring map $\pc_\alpha \mapsto \det M_\alpha$) is an isomorphism.
\end{theorem}

\begin{proof}
Recall that $\openupperschubert$ is parameterized by $k \times n$ matrices $N$ in reverse reduced row echelon form, with submatrix $N_\gamma = I$. We identify $\paramrich$ with the subscheme of $\openupperschubert$ whose minors satisfy the defining conditions of $\paramrich$, i.e. $\det N_{\delta_i}$ is invertible for $i=1, \ldots, k-1$, and $\det N_\alpha = 0$ for $\alpha \notin [\beta, \gamma]$.  With this identification, $\phi$ simply maps $M \in \paramrichmat$ to its reverse reduced row echelon form, which is $\phi(M) = M_\gamma^{-1} M$.  Since $M_\gamma$ is lower unipotent, the entries of $\phi(M)$ are polynomials in the entries of $M$.  Since the minors of any matrix $M \in \paramrichmat$ satisfy the defining conditions of $\paramrich$, we have $\phi(M) \in \paramrich$.

We now describe the inverse map  $\psi : \paramrich \to \paramrichmat$.  Let $N$ be as above, and consider the square submatrix $N_\beta$.  The condition that $\det N_{\delta_i}$ is invertible for $i=1, \ldots, k-1$ asserts (since $N_\gamma = I$) that the top-left-justified minors of $N_\beta$ are nonzero. Therefore $N_\beta$ has a unique decomposition $N_\beta = L D U$ with $L$ lower unipotent, $U$ upper unipotent, and $D$ invertible diagonal. The entries of these matrices are rational functions of $N$, with only the $\det N_{\delta_i}$ as denominators.  We can therefore define $\psi(N) = L^{-1}N$.  

We show that $\psi(N) \in \paramrich$.  Let $M = \psi(N)$.
Immediately, we have that $M_{i, \beta(i)}$ is invertible, and $M_{i, \gamma(i)} = 1$ because $M_\beta = L^{-1}N_\beta = DU$ and $M_\gamma = L^{-1}N_\gamma = L^{-1}$. 
The conditions $M_{i,j} = 0$ if $j > \gamma(i)$ hold because they hold for $M$ and $L^{-1}M$ is obtained from $M$ by downwards row operations.
To that see $M_{i,j} = 0$ for $j < \beta(i)$, first note that if $j \in \beta$, this is true again because $M_{\beta} = DU$.  Now, fix $i$ and $j \notin \beta$, and assume for induction that $M_{i',j} = 0$ for all $i' > i$.  Then for $\alpha = \beta \cup \{j\} \setminus \beta(i)$, we have (by expanding the determinant)
\[
  M_{i,j} = \pm \det M_\alpha \cdot 
   \big(\prod_{i' \neq i} M_{i', \beta(i')}\big)^{-1}
\,.
\]
When $j< \beta(i)$, $\alpha \notin [\beta,\gamma]$, and therefore $\det M_\alpha = \det N_\alpha = 0$; hence $M_{i,j} = 0$.

Finally, if we assume either $M = \psi(N)$ or $N = \phi(M)$, we have $M_\gamma = L^{-1}$, which shows that $\phi$ and $\psi$ are mutually inverse isomorphisms.
\end{proof}

\begin{corollary}
\label{cor:localizedufd}
$\paramrichring$ is an unique factorization domain.
\end{corollary}

%%%%%%%%%%%%%
%%%%%%%%%%%%%
\section{Some implications of the \Plucker relations}

To complete the proof of Proposition~\ref{prop:UFD}, we must show that $\pc_{\delta_t}$ is either a unit or a prime of $\openrichring$. To this end, we study the principal ideal $\langle \pc_{\delta_t} \rangle \subseteq \openrichring$. We begin by identifying other \Plucker coordinates that must be in this ideal. Let
\[
\calP_t 
= \{\alpha \in [\beta, \gamma] \mid \alpha \cap [\beta(t+1), \gamma(t)] 
\neq \emptyset\}
\,,
\]
and
\[
  \overline\calP_t = [\beta, \gamma] \setminus \calP_t
\,.
\]
Note that $\delta_t \in \overline\calP_t$. Also note that the interval $[\beta(t+1), \gamma(t)]$, which appears in definition of $\calP_t$, is empty if $\beta(t+1) > \gamma(t)$; in this case $\calP_t$ is also empty.

\begin{lemma}
\label{lem:principalideal}
If $\alpha \in \overline\calP_t$, then $\pc_\alpha \in \langle \pc_{\delta_t} \rangle$.
\end{lemma}

\begin{proof}
First, note that if $\alpha \in \overline\calP_t$, we must have $\alpha(1) < \dots < \alpha(t) < \beta(t+1)$ and $\gamma(t) < \alpha(t+1) < \dots < \alpha(k)$. We define a new partial order $\preceq_t$ on $\overline\calP_t$: if $\alpha, \alpha' \in \overline\calP_t$, we put $\alpha \preceq_t \alpha'$ if $\alpha(i) \leq \alpha'(i)$ for $i=1, \dots, t$, and $\alpha(i) \geq \alpha'(i)$ for $i=t+1, \dots, k$. This makes $\delta_t$ the unique minimal element of $\overline\calP_t$ with respect to $\preceq_t$.

Let $\calQ_t = \{\alpha \in \overline\calP_t \mid
\pc_\alpha \notin \langle \pc_{\delta_t} \rangle\}$. We must show
that $\calQ_t$ is empty.
Suppose to the contrary that $\calQ_t \neq \emptyset$, and take
$\alpha \in \calQ_t$ which is minimal with respect to $\preceq_t$.
Certainly $\alpha \neq \delta_t$, so we have 
either 
\begin{packedenumi}
\item $\alpha(i) > \beta(i)$ for some $i \leq t$, or 
\item $\alpha(i) <\gamma(i)$ for some $i \geq t+1$.
\end{packedenumi}

Suppose (i) is true.  Take $i$ to be the smallest integer such
that $\alpha(i) > \beta(i)$, and consider the \Plucker relation
\eqref{eqn:pluckerrelation}. Notice that for all $j \in \alpha \setminus 
\beta$, we have either
$\alpha {\setminus} \{j\} {\cup} \{i\} \prec_t \alpha$ or
$\alpha {\setminus} \{j\} {\cup} \{i\} \notin [\beta,\gamma]$.
In either case, we have
$\pc_{\alpha \setminus \{j\} \cup \{i\}} \in 
\langle \pc_{\delta_t}\rangle$ (in case (i) because of the
minimality of $\alpha$, and in case (ii) because this \Plucker 
coordinate is already zero in $\openrichring$).  
Since $\pc_\beta$ is
a unit in $\openrichring$, 
we deduce from~\eqref{eqn:pluckerrelation} that
$\pc_\alpha \in \langle \pc_{\delta_t}\rangle$.  Hence 
$\alpha \notin \calQ_t$, for a contradiction.

If (ii) is true, a similar argument applies.  Take $i$ to be the largest 
integer such that $\alpha(i) < \gamma(i)$, and use 
$\gamma$ in place of $\beta$ in the \Plucker relation.
\end{proof}

\begin{corollary}
\label{cor:unit}
If $\calP_t$ is empty, then $\pc_{\delta_t}$ is a unit in 
$\openrich$.
\end{corollary}

\begin{proof}
By Lemma~\ref{lem:principalideal}, if $\calP_t$ is empty, we have
$\pc_\alpha \in \langle \pc_{\delta_t} \rangle$ for all 
$\alpha \in [\beta, \gamma]$.  In particular, 
$\pc_\beta \in \langle \pc_{\delta_t} \rangle$. As $\pc_\beta$ is a
unit, this implies that $\pc_{\delta_t}$ is also a unit.
\end{proof}

%%%%%%%%%%%%%
%%%%%%%%%%%%%
\section{Positroids}

We now show that when $\calP_t$ is non-empty, it is (the set of bases of) a positroid.  The concept of a positroid comes from the theory of total non-negativity: a \emph{positroid} is a matroid that is representable by a totally non-negative matrix. However, we will not need this definition. There are many equivalent characterizations of positroids; for our purposes, Theorem~\ref{thm:bruhat} below may be taken as a definition.

Let $\symgroup$ denote the symmetric group of permutations of $[n]$, generated by the simple transpositions $s_1, \dots, s_{n-1} \in \symgroup$. There is a natural projection map $\pi_k : \symgroup \to \allsubsets$, given by
\[
  \pi_k(w) = \{w(1), \dots, w(k)\}
\,.
\]
The \emph{Bruhat order} on $\symgroup$ is the partial order characterized by the fact that $u \leq v$ iff $\pi_j(u) \leq \pi_j(v)$ for all $j \in [n]$.

\begin{theorem}[Postnikov \cite{Postnikov}]
\label{thm:bruhat}
A subset $\calM \subseteq \allsubsets$ is a positroid if and only if $\calM = \pi_k [u,v]$ for some $u, v \in \symgroup$, $u \leq v$.
\end{theorem}

A permutation $w \in \symgroup$ is a \emph{Grassmannian permutation} if $w(1) < \dots < w(k)$, and $w(k+1) < \dots < w(n)$. For every $\alpha \in \allsubsets$, there is a unique Grassmannian permutation $w_\alpha \in \symgroup$ such that $\pi_k(w_\alpha) = \alpha$.   It has the property that $w_\alpha(i) = \alpha(i)$, for $i=1, \dots, k$.  Futhermore, we have $\pi_k[w_\alpha, w_{\alpha'}] = [\alpha, \alpha']$, for all $\alpha, \alpha' \in \allsubsets$.

\begin{theorem}
\label{thm:positroidset}
$\calP_t = \pi_k [w_\beta s_t, w_\gamma]$.  Hence, $\calP_t$ is either empty or it is a positroid.  
\end{theorem}

\begin{proof}
Suppose $\alpha \in \pi_k[w_\beta s_t, w_\gamma]$.  
Since $[w_\beta s_t, w_\gamma] \subseteq [w_\beta, w_\gamma]$,
we have that $\alpha \in \pi_k [w_\beta, w_\gamma] = [\beta, \gamma]$.
Next, write $\alpha = \pi_k(u)$, $u \in [w_\beta s_t, w_\gamma]$.
Note that $\pi_t(w_\beta s_t) = 
\{\beta(1) , \dots , \beta(t-1) , \beta(t+1)\}$,  and
$\pi_t(w_\gamma) = \{\gamma(1) , \dots , \gamma(t)\}$.  It follows that
\[
\{\beta(1) , \dots , \beta(t-1) , \beta(t+1)\} \leq
\pi_t(u) 
\leq
\{\gamma(1), \dots, \gamma(t-1), \gamma(t)\}
\,.
\]
Therefore, $\beta(t+1) \leq u(i) \leq \gamma(t)$ for some $i \leq t$.
Since $u(i) \in \alpha$, we deduce that
$\alpha \cap [\beta(t+1), \gamma(t)] \neq \emptyset$.
Since $\alpha \in [\beta,\gamma]$ and 
$\alpha \cap [\beta(t+1), \gamma(t)] \neq \emptyset$, we have
$\alpha \in \calP_t$.

Now, suppose $\alpha \in \calP_t$. 
Since $\alpha(t) \leq \gamma(t)$,
if $\alpha(t) \notin [\beta(t+1),\gamma(t)]$ we must have
$\alpha(1) < \dots < \alpha(t) < \beta(t+1)$.  Similarly, since
$\alpha(t+1) \geq \beta(t+1)$, if $\alpha(t+1) \notin
[\beta(t+1),\gamma(t)]$ we must have 
$\gamma(t) > \alpha(t+1) > \dots > \alpha(k)$.  Since 
$\alpha \cap [\beta(t+1), \gamma(t)] \neq \emptyset$,
we must have either 
\begin{packedenumi}
\item $\alpha(t) \in [\beta(t+1), \gamma(t)]$, or
\item $\alpha(t+1) \in [\beta(t+1), \gamma(t)]$.  
\end{packedenumi} From here, one can easily check 
that in case (i), $w_\alpha \in [w_\beta s_t, w_\gamma]$, and in case (ii),
$w_\alpha s_t \in [w_\beta s_t, w_\gamma]$.  Since
$\pi_k(w_\alpha) = \pi_k(w_\alpha s_t) = \alpha$, this shows that
$\alpha \in \pi_k [w_\beta s_t, w_\gamma]$.
\end{proof}

%%%%%%%%%%%%%
%%%%%%%%%%%%%
\section{Positroid varieties}

For $\calM \subseteq \allsubsets$, we can consider the scheme defined by the vanishing of \Plucker coordinates not in $\calM$:
\[
    \matroidvariety{\calM} = 
    \Proj \grring/\langle \pc_\alpha \mid \alpha \notin \calM \rangle
\,.
\]
Such schemes are unions of strata in the GGMS stratification of the Grassmmannian \cite{GGMS}; in general, they can be quite badly behaved, even if $\calM$ is a matroid (see e.g. discussion in \cite[\S1]{KLS}). However, when $\calM$ is a  positroid $\matroidvariety{\calM}$ has many nice properties, including the following.

\begin{theorem}[Knutson--Lam--Speyer \cite{KLS}]
\label{thm:KLS}
If $\calM \subseteq \allsubsets$ is a positroid, then $\matroidvariety{\calM}$ is an integral scheme.
\end{theorem}

When $\calM$ is a positroid, $\matroidvariety{\calM}$ is called a \emph{positroid variety}.   Since every interval $[\beta, \gamma] \subseteq \allsubsets$ is a positroid, every Richardson variety is a positroid variety. Applying Theorem~\ref{thm:KLS} to our situation, we obtain the following.

\begin{theorem}
\label{thm:positroiddivisor}
Suppose $\calP_t$ is non-empty. $\pdivisor$ is a non-empty closed integral subscheme of $\rich$ of codimension $1$. Its interesection with the open Richardson variety, $\openpdivisor  = \pdivisor \cap \openrich$, is a non-empty open subscheme of $\pdivisor$ and a closed integral subcheme of $\openrich$.  Moreover, 
\begin{equation}
\label{eqn:positroiddivisor}
  \openpdivisor
  =
  \Proj \openrichring / \langle \pc_{\delta_t} \rangle
\,.
\end{equation}
\end{theorem}

\begin{proof}
The fact that $\pdivisor$ is integral follows from  Theorems~\ref{thm:positroidset} and~\ref{thm:KLS}. Note that in the notation above, $\rich = \matroidvariety{[\beta,\gamma]}$. Since $\calP_t \subseteq [\beta,\gamma]$, $\pdivisor$ is a closed subscheme of $\matroidvariety{[\beta,\gamma]}$. Since $\calP_t$ is not contained in any proper subinterval of $[\beta,\gamma]$, $\pdivisor$ is not properly contained in any of the boundary components of $\openrich$, and therefore the intersection with $\openrich$ is non-empty. The fact that the codimension is $1$ follows from~\eqref{eqn:positroiddivisor}, as $\langle \pc_{\delta_t} \rangle$ is a principal ideal.

As for proving \eqref{eqn:positroiddivisor}, first note that both sides are subschemes of the Grassmannian defined by the vanishing and non-vanishing of certain \Plucker coordinates. $\openpdivisor$ is defined by the vanishing of $\{\pc_\alpha \mid \alpha \in \allsubsets \setminus \calP_t\}$, and the non-vanishing of $\{\pc_\beta, \pc_\gamma\}$. $\Proj \openrichring / \langle \pc_{\delta_t} \rangle$ is defined by the vanishing of $\{\pc_\alpha \mid \alpha \in [\beta, \gamma] \cup \{\pc_{\delta_t}\}\}$, and the non-vanishing of $\{\pc_\beta, \pc_\gamma\}$. Since $[\beta, \gamma] \cup \{\pc_{\delta_t}\} \subseteq \allsubsets 
\setminus \calP_t$, we have
\[
  \openpdivisor 
  \subseteq
  \Proj \openrichring / \langle \pc_{\delta_t} \rangle
\,.
\]
By Lemma~\ref{lem:principalideal}, all \Plucker coordinates which vanish on $\pdivisor$ also vanish on $\Proj \openrichring / \langle \pc_{\delta_t} \rangle$, so we also have the reverse containment.  
\end{proof}

\begin{corollary}
\label{cor:prime}
If $\calP_t$ is non-empty then $\pc_{\delta_t}$ is a prime element of $\openrichring$.
\end{corollary}

\begin{proof}
By Theorem~\ref{thm:positroiddivisor}, $\Proj \openrichring / \langle \pc_{\delta_t} \rangle$ is a closed integral subscheme of $\openrich = \Proj \openrichring$, which is equivalent.
\end{proof}

%%%%%%%%%%%%%
%%%%%%%%%%%%%
\section{Remarks}
\label{sec:remarks}

As an additional consequence of Theorems~\ref{thm:paramrich} and~\ref{thm:positroiddivisor}, we obtain a description of the nicely parametrizable scheme $\paramrich$, as the complement of a union of postroid varieties inside the Richardson variety $\rich$.  Let
\begin{gather*} 
\Sigma_0 = \{\calP_t \mid \calP_t \neq \emptyset,\,1 \leq t \leq k-1\} \\
\Sigma_1 = \{[\beta', \gamma] \mid \beta \lessdot \beta' < \gamma\}~ \\
\Sigma_2 = \{[\beta, \gamma'] \mid \beta < \gamma' \lessdot \gamma\}\,. 
\end{gather*}
Here $\lessdot$ denotes the covering relation in the poset $\allsubsets$. Let $\Sigma = \Sigma_0 \cup \Sigma_1 \cup \Sigma_2$.  Each element of $\Sigma$ is a postroid. Then we have
\[
   \paramrich = 
  \rich \setminus \bigcup_{\calM \in \Sigma} \matroidvariety{\calM}
\,.
\]
Note that for $\calM \in \Sigma_1$ or $\calM \in \Sigma_2$, the positroid variety $\matroidvariety{\calM}$ is a Richardson variety: these are the boundary components of $\openrich$.

The varieties $\pdivisor$ admit a couple of other nice descriptions. In general positroid varieties are projections of Richardson varieties in the full flag manifold. However, as $\calP_t = \pi_k[w_\beta s_t, w_\gamma]$ and $w_\beta s_t$ is a permutation with at most two descents, $\pdivisor$ is in fact a projection of a Richardson variety in the two-step flag manifold $\mathrm{Fl}(t,k;n)$.

Alternatively, $\pdivisor$ can be described as the intersection three permuted Schubert varieties.  $\symgroup$ acts on $k \times n$ matrices by permuting the columns, and therefore acts on $\Gr(k,n)$.  Let $c \in \symgroup$ be the long cycle, $c = (1\ 2\ 3\  \dots\ n)$, and let $w_0 \in \symgroup$ be the long element, $w_0(i) = n+1-i$. Acting by $w_0$ on a Schubert variety gives an opposite Schubert variety. Acting by $c^j$ gives a \emph{cyclically shifted Schubert variety}. Every positroid variety can be expressed as an intersection of cyclically shifted Schubert varieties (and/or opposite Schubert varieties)
\cite{KLS}.
In our case, let $\calI_t = \{\alpha \in \allsubsets \mid \alpha \cap 
[\beta(t+1), \gamma(t)] \neq \emptyset$\}.
Then $\pdivisor = \rich \cap \matroidvariety{\calI_t}$. Moreover, $\matroidvariety{\calI_t}$ is a cyclically shifted Schubert variety: if $\epsilon = \{1,2,3, \dots k-1, n-\gamma(t)+\beta(t+1)\}$, then
$\matroidvariety{\calI_t} = c^{\gamma(t)} X_\epsilon$.
Thus
\[
  \pdivisor = \lowerschubert \cap \upperschubert \cap c^{\gamma(t)} X_\epsilon
\,.
\]
We remark that $X_\epsilon$ is in fact a \emph{special} Schubert variety (the partition associated to $\epsilon$ is a single row); these have many nice properties and play a fundamental role in Schubert calculus (see e.g. \cite{Fulton}).

%%%%%%%%%%%%%
%%%%%%%%%%%%%
\section{Acknowledgements}

We wish to thank Allen Knutson, Cameron Marcott, David Speyer,
and Alex Woo, for helpful conversations and correspondence pertaining to
this paper.  Research of the second author was supported by 
NSERC Discovery Grant RGPIN-04741-2018.

%%%%%%%%%%%%%
%%%%%%%%%%%%%

\end{document}